\documentclass[leqno,12pt,a4paper]{amsart}
\usepackage{graphicx}
\usepackage[usenames,dvipsnames]{color}
\usepackage{hyperref}
\usepackage{subfigure}
\usepackage{array, ragged2e}
\newtheorem{thm}{Theorem}[section]
\newtheorem{lemma}[thm]{Lemma}
\newtheorem{cor}[thm]{Corollary}
\newtheorem{prop}[thm]{Proposition}
\theoremstyle{definition}

\newtheorem{remark}[thm]{Remark}

\newtheorem{definition}[thm]{Definition}
\newtheorem{conjecture}[thm]{Conjecture}
\numberwithin{equation}{section}
\newcommand{\orig}{\mathbf{0}}
\newcommand{\card}[1]{\mid\!{#1}\!\mid}
\newcommand{\pro}[2]{\langle{#1},{#2}\rangle}
\newcommand{\NQ}{N_\mathbb{Q}}
\newcommand{\bd}{\partial}
\newcommand{\Hom}{\mathrm{Hom}}

\newcommand{\N}{\mathbb{Z}_{>0}}
\newcommand{\Z}{\mathbb{Z}}
\newcommand{\Q}{\mathbb{Q}}

\newcommand{\C}{\mathbb{C}}
\newcommand{\Proj}{\mathbb{P}}
\newcommand{\V}[1]{\mathcal{V}\!\left({#1}\right)}
\newcommand{\E}[1]{\mathcal{E}\!\left({#1}\right)}
\newcommand{\F}[1]{\mathcal{F}\!\left({#1}\right)}

\newcommand{\dual}{*}
\newcommand{\Lambdadual}{\Lambda^*}
\newcommand{\LambdaP}{\Lambda_P}
\newcommand{\LambdaV}[1]{\Lambda_{\mathcal{V}\left({#1}\right)}}
\newcommand{\LambdaE}[1]{\Lambda_{\mathcal{E}\left({#1}\right)}}
\newcommand{\Rcal}{\mathcal{R}}
\newcommand{\magma}{{\sc Magma}}
\newcommand{\intr}[1]{\mathrm{int}\!\left({#1}\right)}
\newcommand{\Vol}[1]{\mathrm{Vol}\!\left({#1}\right)}
\newcommand{\conv}[1]{\mathrm{conv}\!\left({#1}\right)}
\newcommand{\sconv}[1]{\mathrm{conv}\!\left\{{#1}\right\}}
\newcommand{\abs}[1]{\left\vert{#1}\right\vert}
\newcommand{\mult}[1]{\mathrm{mult}\,{#1}}
\renewcommand{\Re}[1]{\mathrm{Re}\!\left(#1\right)}
\newcommand{\lcm}[1]{\mathrm{lcm}\!\left\{#1\right\}}
\newcommand{\modb}[1]{~\mathrm{(mod}\,#1\mathrm{)}}
\renewcommand{\gcd}[1]{\mathrm{gcd}\!\left\{{#1}\right\}}

\newenvironment{ack}{\bigskip\noindent\textbf{Acknowledgments.}}{}
\graphicspath{{images/}}
\begin{document}
\author[A.~M.~Kasprzyk]{Alexander M.~Kasprzyk}
\address{Department of Mathematics\\Imperial College London\\London, SW$7$\ $2$AZ\\UK}
\email{a.m.kasprzyk@imperial.ac.uk}
\author[B.~Nill]{Benjamin Nill}
\address{Department of Mathematics\\Case Western Reserve University\\Cleveland\ OH\ $44106$\\USA}
\email{bnill@math.uga.edu}
\title[Reflexive polytopes of higher index and the number $12$]{Reflexive polytopes of higher index\\and the number $12$}
\begin{abstract}
We introduce reflexive polytopes of index $l$ as a natural generalisation of the notion of a reflexive polytope of index $1$. These $l$-reflexive polytopes also appear as dual pairs. In dimension two we show that they arise from reflexive polygons via a change of the underlying lattice. This allows us to efficiently classify all isomorphism classes of $l$-reflexive polygons up to index $200$. As another application, we show that any reflexive polygon of arbitrary index satisfies the famous ``number $12$'' property. This is a new, infinite class of lattice polygons possessing this property, and extends the previously known sixteen instances. The number $12$ property also holds more generally for $l$-reflexive non-convex or self-intersecting polygonal loops. We conclude by discussing higher-dimensional examples and open questions.
\end{abstract}
\maketitle
\begin{center}
\vspace{1em}
\emph{Dedicated to the memory of Maximilian Kreuzer.}
\vspace{1em}
\end{center}
\section{Introduction and main results}
\subsection{Notation} 
We begin by recalling some basic definitions, and by fixing our notation.

Let $N\cong\Z^n$ be a lattice, and let $P\subset\NQ:=N\otimes_\Z\Q$ be an $n$-dimensional lattice polytope; i.e.~the set of vertices of $P$, denoted by $\V{P}$, is contained in the lattice $N$. We denote the interior of $P$ by $\intr{P}$ and its boundary by $\bd P$. The set of facets (codimension one faces) of $P$ is referred to by $\F{P}$. The \emph{volume of $P$} will always mean the normalised volume $\Vol{P}$ with respect to the ambient lattice $N$. Two lattice polytopes $P\subseteq \NQ$ and $P'\subseteq N'_\Q$ are isomorphic if there exists an affine lattice isomorphism $N\cong N'$ mapping $\V{P}$ onto $\V{P'}$.

A lattice point $x$ in $N\setminus\left\{\orig\right\}$ is \emph{primitive} if the line segment joining $x$ and $\orig$ contains no other lattice points. We denote by $M$ the dual lattice $\Hom(N,\Z)$ of $N$. Given a facet $F$ of $P$ we define its \emph{primitive outer normal} to be the unique primitive lattice point $u_F\in M$ such that $F=\left\{x\in P\mid\pro{u_F}{x}=l_F\right\}$ for some (uniquely determined) $l_F\in\N$. We call $l_F$ the \emph{local index} of $F$; it is equal to the integral distance of $\orig$ from the affine hyperplane spanned by $F$.

\subsection{Reflexive polytopes of higher index}
Reflexive polytopes were first introduced by Batyrev in~\cite{Bat94} in the context of Mirror Symmetry. In subsequent years they were intensively studied and classified as important examples of Fano varieties in toric geometry, and used for the construction of Calabi-Yau varieties (e.g.~\cite{BB96a,BB96b,KS97,KS00,Nil05,Cas06,Oeb07}). They are also intimately connected to commutative algebra and combinatorics via the study of Gorenstein polytopes~\cite{BB97,Ath05,BrR07,BN08}. Here we present a natural generalisation of this setting, 
which we hope will be put to good use in future applications.

\begin{definition}
A lattice polytope $P$ is called \emph{$l$-reflexive} if, for some $l\in\N$, the following conditions hold:
\begin{enumerate}
\item $P$ contains the origin in its (strict) interior;
\item the vertices of $P$ are primitive;
\item for any facet $F$ of $P$ the local index $l_F$ equals $l$.
\end{enumerate}
We also refer to $P$ as a \emph{reflexive polytope of index $l$}.
\end{definition}

The $1$-reflexive polytopes are precisely the reflexive polytopes of~\cite{Bat94}. Note that the requirement that the vertices are primitive prevents multiples of $1$-reflexive polytopes from being $l$-reflexive.

\subsection{Duality}
Let $P\subseteq\NQ$ be a full-dimensional lattice polytope. The dual polyhedron
$$P^\dual:=\left\{y\in M_\Q\mid\pro{y}{x}\leq 1\right\}$$
is a (not necessarily lattice) polytope if and only if $\orig$ lies in the interior of $P$. It is a well-known characterisation of reflexive polytopes that $P$ is reflexive if and only if $P^\dual$ is a lattice polytope (see~\cite{Bat94}). Clearly $P^\dual$ is also a reflexive polytope.

This characterisation has a natural reformulation for $l$-reflexive polytopes:

\begin{prop}\label{prop:duality}
Let $P$ be a lattice polytope with primitive vertices, such that $P$ contains the origin in its interior. Then $P$ is $l$-reflexive if and only if $lP^\dual$ is a lattice polytope having only primitive vertices. In this case, $lP^\dual$ is also $l$-reflexive. This induces a natural duality for $l$-reflexive polytopes:
\begin{equation}\label{eq:correspondence}
P\ \longleftrightarrow\ lP^\dual.
\end{equation}
\end{prop}
\begin{proof}
Note that the vertices of $P^\dual$ are precisely the points $u_F/l_F$, for each facet $F$ of $P$. Analogously, the facets of $P^\dual$ are in one-to-one correspondence with the vertices of $P$.

If $P$ is $l$-reflexive then $lP^\dual=\left\{u_F\mid F\in\F{P}\right\}$, so any vertex of $lP^\dual$ is primitive. Moreover, any facet of $lP^\dual$ is given as $\left\{x\in lP^\dual\mid\pro{v}{x}=l\right\}$ for some $v\in\V{P}$. Since the vertices of $P$ are primitive, it follows that any facet of $lP^\dual$ has local index $l$. Hence $lP^\dual$ is also $l$-reflexive.

Conversely, suppose that $lP^\dual$ is a lattice polytope having only primitive vertices. Then for any facet $F\in\F{P}$ we see that $l(u_F/l_F)$ is a primitive lattice point. Hence $l_F=l$ and $P$ is $l$-reflexive.

Finally, since
$$l(lP^\dual)^\dual=l(\frac{1}{l}P)=P,$$
the duality~\eqref{eq:correspondence} follows by symmetry.
\end{proof}

Notice that when $l=1$ we recover the usual duality of reflexive polytopes.

\subsection{Finiteness and classification}\label{sec:finite}
A reflexive polytope (of index $1$) does not contain any lattice points besides the origin in its interior. This is, in general, not true for reflexive polytopes of higher index. However, it follows from their definition that an $l$-reflexive polytope $P\subseteq\NQ$ satisfies
$$\intr{P/l}\cap N=\left\{\orig\right\},$$
or, equivalently, that $\abs{\intr{P}\cap lN}=1$. A result of Lagarias and Ziegler~\cite{LZ91} implies that, for fixed dimension $n$ and index $l$, there are only \emph{finitely many} isomorphism classes of $n$-dimensional $l$-reflexive polytopes. 

In dimension one there are no $l$-reflexive polytopes when $l>1$, and only one when $l=1$: the line segment $[-1,1]$ corresponding to $\Proj^1$. In dimension two the $l$-reflexive polygons form a subset of the LDP-polygons studied in~\cite{KKN08}. The reader is invited to try to find some $l$-reflexive polygons before proceeding. Whilst it is not too difficult to find all sixteen $1$-reflexive polygons (draw any convex lattice polygon with no interior lattice points other than the origin), it is actually quite challenging to find examples of higher index. For instance, one quickly suspects that there is no reflexive polygon of index $2$. Even more is true, as will be explained in Section~\ref{sec:dim2}.

\begin{prop}\label{prop:odd}
There is no $l$-reflexive polygon of even index.
\end{prop}

There is precisely one $3$-reflexive polygon, which we denote by $P_3$ and is illustrated in Figure~\ref{fig:P3}. This example generalises to a family of $l$-reflexive polygons, one for each odd index. Let $P_l$ be the polygon defined by the convex hull of $\left\{\pm(0,1), \pm(l,2), \pm(l,1)\right\}$. This is a centrally-symmetric hexagon and, since $l$ is odd, $P_l$ is an $l$-reflexive polygon. To see this note that the vertices of $P_l^\dual$ are given by $\left\{\pm(\frac{1}{l},0), \pm(\frac{2}{l},-1), \pm(\frac{1}{l},-1)\right\}$. Hence $P_l$ is actually self-dual in the sense that it is isomorphic to $lP_l^\dual$.

\begin{figure}[htbp]
\centering
\includegraphics[scale=0.8]{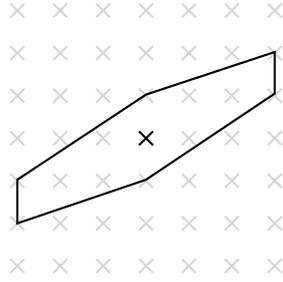}
\caption{The unique $3$-reflexive polygon $P_3$.}
\label{fig:P3}
\end{figure}

Notice that $P:=\sconv{\pm(0,1),\pm(1,1),\pm(1,0)}$ is the unique $1$-reflexive polygon which is also a centrally-symmetric hexagon, and that $P_l=\varphi(P)$, where the map $\varphi$ is given by right multiplication with the matrix
$$\begin{pmatrix}
l&1\\
0&1
\end{pmatrix}.
$$
In Corollary~\ref{cor:reflexive} we will generalise this observation to any $l$-reflexive polygon. This gives a fast classification algorithm, which we implemented in {\magma} (see Appendix~\ref{apdx:source_code}). 

\begin{thm}\label{thm:classification}
For each positive odd integer $l$ let $n(l)$ be the number of isomorphism classes of $l$-reflexive polygons. Then, for $1\leq l<60$:
\begin{center}
\vspace{0.75em}
\begin{tabular}{|r|c|c|c|c|c|c|c|c|c|c|c|c|c|c|c|}\hline
$l$&$1$&$3$&$5$&$7$&$9$&$11$&$13$&$15$&$17$&$19$&$21$&$23$&$25$&$27$&$29$\\\hline
$n(l)$&$16$&$1$&$12$&$29$&$1$&$61$&$81$&$1$&$113$&$131$&$2$&$163$&$50$&$2$&$215$\\\hline\hline
$l$&$31$&$33$&$35$&$37$&$39$&$41$&$43$&$45$&$47$&$49$&$51$&$53$&$55$&$57$&$59$\\\hline
$n(l)$&$233$&$2$&$34$&$285$&$3$&$317$&$335$&$2$&$367$&$182$&$3$&$419$&$72$&$4$&$469$\\\hline
\end{tabular}
\vspace{0.75em}
\end{center}

A complete classification of the $l$-reflexive polygons up to index $200$ is available online via the Graded Ring Database:
\begin{center}
{\rm\href{http://grdb.lboro.ac.uk/forms/toriclr2}{\texttt{http://grdb.lboro.ac.uk/forms/toriclr2}}}
\end{center}
\end{thm}

Lattice polygons with primitive vertices exhibit peculiar behaviours which have a number-theoretic flavour~\cite{Dai09}. Corollary~\ref{cor:quotient} in Section~\ref{sec:dim2} implies the following upper bound on the growth of $l$-reflexive polygons:

\begin{cor}\label{cor:phi}
There are at most $16(\phi(l)-1)$ isomorphism classes of $l$-reflexive polygons, where $\phi$ is Euler's totient function.
\end{cor}

Table~\ref{tab:3k_growth} illustrates the astonishingly slow growth in the number of non-isomorphic $3k$-reflexive polygons. Proposition~\ref{prop:3k_reflexive} states that these are always self-dual hexagons.

\begin{table}[htdp]
\caption{The number of isomorphism classes of the $3k$-reflexive polygons up to index $111$ (for $k$ odd).}
\centering
\begin{tabular}{|r|c|c|c|c|c|c|c|c|c|c|c|c|c|c|c|c|c|c|c|}\hline
$k$&$1$&$3$&$5$&$7$&$9$&$11$&$13$&$15$&$17$&$19$&$21$&$23$&$25$&$27$&$29$&$31$&$33$&$35$&$37$\\\hline
$n(3k)$&$1$&$1$&$1$&$2$&$2$&$2$&$3$&$2$&$3$&$4$&$3$&$4$&$3$&$5$&$5$&$6$&$5$&$3$&$7$\\\hline
\end{tabular}
\label{tab:3k_growth}
\end{table}

Our observations in the database show that the number of self-dual $l$-reflexive polygons grows very slowly (see Table~\ref{tab:self_dual}). It would be interesting to make this precise.

\begin{table}[htdp]
\caption{The number of isomorphism classes $s(l)$ of the self-dual $l$-reflexive polygons up to index $59$.}
\centering
\begin{tabular}{|r|c|c|c|c|c|c|c|c|c|c|c|c|c|c|c|}\hline
$l$&$1$&$3$&$5$&$7$&$9$&$11$&$13$&$15$&$17$&$19$&$21$&$23$&$25$&$27$&$29$\\\hline
$s(l)$&$4$&$1$&$4$&$3$&$1$&$3$&$7$&$1$&$7$&$5$&$2$&$5$&$6$&$2$&$9$\\\hline\hline
$l$&$31$&$33$&$35$&$37$&$39$&$41$&$43$&$45$&$47$&$49$&$51$&$53$&$55$&$57$&$59$\\\hline
$s(l)$&$7$&$2$&$4$&$11$&$3$&$11$&$9$&$2$&$9$&$8$&$3$&$13$&$6$&$4$&$11$\\\hline
\end{tabular}
\label{tab:self_dual}
\end{table}

\subsection{The ``number $12$''}
Recall that the famous ``number $12$'' property~\cite{Ful93,PRV00,HS02,HS09} states that the sum of the number of boundary lattice points on a $1$-reflexive polygon and the number of boundary lattice points on its dual always equals twelve. In Section 2 we extend this property to the class of reflexive polygons of higher index. While there are only sixteen reflexive polygons of index $1$ up to isomorphisms \cite{PRV00}, there are infinitely many reflexive polytopes of higher index.

\begin{thm}\label{thm:number12}
Let $P$ be a $l$-reflexive polygon. Then 
$$\abs{\bd P \cap N} + \abs{\bd (lP^\dual) \cap M} = 12.$$
\end{thm}

Note that $P$ and $l P^\dual$ have the same number of vertices. 

\begin{cor}
Any reflexive polygon of arbitrary index has at most nine boundary points and six vertices.
\end{cor}

\begin{cor}
Any (possibly singular) toric del Pezzo surface whose automorphism group acts transitively on the set of torus-invariant points has at most six torus-invariant points.
\end{cor}

Our proof of Theorem~\ref{thm:number12} involves a purely combinatorial argument which reduces 
the statement to the ``number 12'' property for $1$-reflexive polygons. This classical statement has a very elegant algebro-geometric proof (see~\cite{PRV00}). We wonder whether there is also a direct argument arising from algebraic geometry in the case of $l$-reflexive polygons.

\subsection{Log del Pezzo surfaces}
A log del Pezzo surface $X$ is a normal complex surface with ample $\Q$-Cartier anticanonical divisor $-K_X$ and at worst log terminal singularities. They have been extensively studied by Nukulin, Alexeev, and Nakayama (see, for example,~\cite{AN06,Nak07}). If we restrict our attention to the toric case then there exists a bijective correspondence with certain lattice polygons, called \emph{LDP-polygons}.

A lattice polygon $P$ is said to be an LDP-polygon if:
\begin{enumerate}
\item $P$ contains the origin in its (strict) interior;
\item the vertices of $P$ are primitive.
\end{enumerate}
The LDP-triangles were first studied by Dais in~\cite{Dais07}, and LDP-polygons in~\cite{DN08}. Upper bounds on the volume and number of boundary points of $P$ in terms of the index $l:=\lcm{l_F\mid F\in\F{P}}$, and a technique for classifying the LDP-polygons of given index, were derived in~\cite{KKN08}. Here, $l$ equals the minimum positive integer $k$ such that $-kK_X$ is a Cartier divisor.

Clearly the $l$-reflexive polygons form a special subclass of the LDP-polygons of index $l$. In fact it appears to be relatively unusual for an LDP-polygon to be $l$-reflexive; for example, there are $1142$ LDP-polygons of index $13$~(\cite[Theorem~1.2]{KKN08}), of which only $7\%$ are $l$-reflexive. Furthermore, in contrast with $l$-reflexive polygons, there exist LDP-polygons of even index.

\subsection{Organisation of the paper}
In Section~\ref{sec:dim2} we investigate $l$-reflexive polygons and consider a non-convex generalisation. We prove Theorem~\ref{thm:number12} and Proposition~\ref{prop:odd}, and give an easy classification algorithm leading to Theorem~\ref{thm:classification}. In Section~\ref{sec:examples} we describe higher-dimensional examples and state some open questions.

\section{Dimension two}\label{sec:dim2}
\subsection{$l$-reflexive loops}
There is a generalisation of reflexive polygons due to Poonen and Rodriguez-Villegas~\cite{PRV00}. We give the analogous definition for higher index.

\begin{definition}Let $x_1, \ldots, x_t \in N=\Z^2$ be non-zero lattice points, and define $x_0:=x_t$, $x_{t+1} := x_1$. 
We say $\{x_1, \ldots, x_t\}$ is the set of \emph{boundary lattice points} $\bd P \cap N$ of an \emph{$l$-reflexive loop} $P$ \emph{of length $t$} if the following three conditions are satisfied for each $i=1,\dots,t$:
\begin{enumerate}
\item the lattice point $x_{i+1}-x_i$ is primitive;
\item the determinant of the $2\times 2$-matrix $A_i$ formed by $x_i, x_{i+1}$ equals $\pm l$;
\item if $x_i$ is a \emph{vertex} (i.e.~$x_i \not\in \sconv{x_{i-1},x_{i+1}}$) then it is primitive.
\end{enumerate}
The \emph{length of $P$} is defined as $\sum_{i=1}^t \det(A_i)/l$. The set of \emph{facets} of $P$ is naturally given as the set of line segments between successive vertices. Note that an $l$-reflexive loop may be a non-convex or self-intersecting polygonal loop (see Figure~\ref{fig:reflexive-loop}).

For $i=1, \ldots, t$, let $u_i$ be the primitive outer normal to the segment $\sconv{x_i,x_{i+1}}$. Then 
$$\bigcup_{i=1}^t \sconv{u_i,u_{i+1}} \cap M$$
is the set of boundary lattice points of the \emph{dual $l$-reflexive loop} $l P^\dual$ (see Figure~\ref{fig:reflexive-loop}). We leave it to the reader to check that this is well-defined and duality as in Proposition~\ref{prop:duality} generalizes to this setting. 
\end{definition}

\begin{figure}[htbp]
\centering
\def\svgwidth{6.5cm}
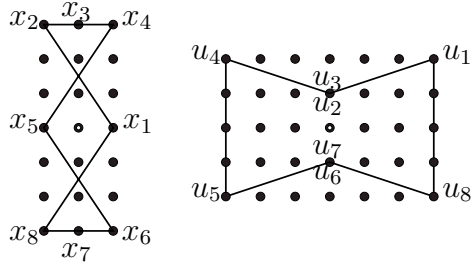
\caption{A $3$-reflexive loop of length $0$ and its dual of length $12$.}
\label{fig:reflexive-loop}
\end{figure}

\subsection{Change of lattice}
Throughout, let $P$ be an $l$-reflexive loop. 

\begin{definition} 
Let $\LambdaP$ be the lattice generated by the boundary lattice points of $P$.
\end{definition}

Here is our main technical result.

\begin{prop}\label{prop:lattice}
Let $P$ be an $l$-reflexive loop. Then 
$$\Lambda_{l P^\dual} = l \Lambdadual_P.$$
Moreover, $\LambdaP \subseteq N$ and $l \Lambdadual_P \subseteq M$ are both lattices of index $l$.
\end{prop}
\begin{proof}
Let us first show that if $x\in\bd P \cap N$ and $u_F\in\V{l P^\dual}$ then $\pro{u_F}{x}\in l \Z$. We prove this by inductively showing that if $F_1, \ldots, F_s$ is a successive sequence of facets of $P$ (in either clockwise or counter-clockwise order) such that $F_i$ is adjacent to $F_{i-1}$ (for $i = 2, \ldots, s$), and if $F_{s+1}$ is the other facet adjacent to $F_s$, then $\pro{u_{F_1}}{x} \in l \Z$ for any point $x \in F_{s+1} \cap N$. 

By definition we only have to consider $s > 1$. We may assume by a unimodular transformation that $u_{F_1} = (0,1)$. Let $u_{F_{s+1}} = (a,b)$, $x =(c,d) \in F_{s+1} \cap N$, and $v = (k,l) \in \V{F_1 \cap F_2}$. Applying the induction hypothesis to $F_{s+1}, F_s, \ldots, F_2$, yields that $l$ divides $\pro{u_{F_{s+1}}}{v} = a k + b l$. Therefore, $l$ divides $a k$. Since $v$ is primitive, $\gcd{k,l} = 1$, hence, $l$ divides $a$. We know that $\pro{u_{F_{s+1}}}{x} = a c + b d = l$. Therefore, $l$ divides $b d$. Since, $u_{F_{s+1}}$ is primitive, $\gcd{a,b} = 1$, so $\gcd{l,b} = 1$, hence $l$ divides $d = \pro{u_{F_1}}{x}$ as desired.

By symmetry we may assume $x \in \intr{F} \cap N$ and $y \in \intr{G} \cap M$, where $F \in \F{P}$ and $G \in \F{l P^\dual}$. Let $u_F \in \V{l P^\dual}$ and $v_G \in \V{P}$ be the corresponding primitive outer normals. We may again assume that $v_G = (0,1)$. Let $u_F = (a,b)$, $x = (c,d)$, $y = (k,l)$. As we have seen, $l$ divides $\pro{v_G}{u_F} = b$. On the other hand, $l = \pro{u_F}{x} = a c + b d$, so $l$ divides $a c$. Since $\gcd{a,b} = 1$, so $\gcd{a,l} = 1$, hence $l$ divides $c$. Therefore, $l$ divides $c k + d l = \pro{x}{y}$. 

This shows $\Lambda_{l P^\dual} \subseteq l \Lambdadual_P$. To show the converse direction, let $F$ be a facet of $P$. We denote by $\Lambda_F$ the lattice generated by the lattice points in $F$. We may assume that $u_F = (0,1)$ and hence $\Lambda_F$ is generated by $(1,0),(0,l)$. For any boundary lattice point $x$ in $P$, we have shown that $l$ divides $\pro{u_F}{x}$, hence $x \in \Lambda_F$. This proves $\LambdaP = \Lambda_F$. In particular, $\LambdaP \subseteq N$ has index $l$. By symmetry, $\Lambda_{l P^\dual} \subseteq M$ also has index $l$, and we have that $\Lambdadual_P$ is generated by $(1,0)$ and $(0,1/l)$. Therefore $l \Lambdadual_P \subseteq M$ is also a lattice of index $l$. Since $\Lambda_{l P^\dual} \subseteq l \Lambdadual_P$ are sublattices of $M$ of the same index, they are equal.
\end{proof}

\subsection{Applications}
We describe several corollaries to Proposition~\ref{prop:lattice}.

\begin{cor}\label{cor:reflexive}
Let $P$ be an $l$-reflexive loop. Then $P$ is a $1$-reflexive loop with respect to the lattice $\LambdaP$, which we call the \emph{$1$-reflexive loop associated to $P$}. Moreover, its dual $1$-reflexive loop is isomorphic to the $1$-reflexive loop associated to $l P^\dual$.
\end{cor}
\begin{proof}
By Proposition~\ref{prop:lattice}, $\Lambda_{l P^\dual}/l = \Lambdadual_P$. Since any vertex of $P^\dual$ is of the form $u/l$ for some (primitive) vertex $u$ of $l P^\dual$, the vertices of $P^\dual$ lie in $\Lambdadual_P$. They are necessarily primitive, since for any vertex $w \in \V{P^\dual}$ there is a vertex $v \in \V{P}$ such that $\pro{w}{v} = 1$. Therefore by Proposition~\ref{prop:duality} $P^\dual$ is a $1$-reflexive loop, say $Q^\dual$, with respect to $\Lambdadual_P$. Moreover, via multiplication by $l$, $Q^\dual$ is isomorphic to $l P^\dual$ with respect to the lattice $\Lambda_{l P^\dual}$.
\end{proof}

Using this we can derive an efficient classification algorithm for $l$-reflexive polygons analogous to an approach by Conrads~\cite{Con02}.

\begin{cor}\label{cor:quotient}
Let $\Rcal$ be a set of representatives of all isomorphism classes of $1$-reflexive polygons (respectively, loops). We may choose any $Q \in \Rcal$ such that $(0,1) \in N$ is a vertex of $Q$ and $(0,1) \in M$ is a vertex of $Q^*$. If $P$ is an $l$-reflexive polygon (respectively, loop) of index $l \geq 2$, then there exists $Q \in \Rcal$ such that $P$ is isomorphic to the image of $Q$ under the map
$$\begin{pmatrix}l & i\\0 & 1\end{pmatrix}$$
for $0 < i < l$ coprime to $l$.
\end{cor}

\begin{proof}
By Proposition~\ref{prop:lattice} there is an isomorphism $\Z^2 \to \LambdaP \subseteq \Z^2$ given by right-multiplying an integer $2 \times 2$-matrix $H'$ of determinant $l$. Corollary~\ref{cor:reflexive} yields that $P=Q'H'$ for some $1$-reflexive polygon (equiv.~loop) $Q'$. Thus, by our assumption, there exists $Q \in \Rcal$ and a unimodular $2 \times 2$-matrix $U'$ such that $Q'=QU'$, hence, $U'H'$ maps $Q$ onto $P$. The Hermite normal form theorem yields that there exists a unimodular $2 \times 2$-matrix $U$ such that $H := U' H' U$ is in upper triangular Hermite normal form
$$\begin{pmatrix}d & i\\0 & l/d\end{pmatrix}$$
for $d$ a divisor of $l$, and $0 \leq i < d$. Therefore $Q$ maps via $H$ onto the image $P'$ of $P$ under $U$. Since $(0,1)$ is a vertex of $Q$, the row vector $(0,l/d)$ is a vertex of the $l$-reflexive polygon (equiv.~loop) $P'$ and hence is primitive. Therefore $l/d=1$. 

Let us consider the dual picture. One checks that $l (P')^*$ is equal to the image of $Q^*$ under the matrix 
$$M := l (H^{\tiny{T}})^{-1} = \begin{pmatrix}1 & 0\\-i & l\end{pmatrix}.$$
Define $g:=\gcd{l,i}\ne 0$. There exist unique integers $j,k$ with $0 \leq j < l/g$ such that
\begin{equation}\label{eq:euclid}
-j i+k l=g.
\end{equation}
Therefore, we can define an integer matrix $J$ with $\det(J)=1$ by setting
$$J := \begin{pmatrix}\frac{l}{g} & j \\ \frac{i}{g} & k\end{pmatrix}.$$
Hence
$$M \;\cdot\; J = \begin{pmatrix}\frac{l}{g} & j \\ 0 & g\end{pmatrix} =: K$$
is in Hermite normal form. Therefore $l(P')^*$ is isomorphic to the image of $Q^*$ under the matrix $K$. As above, our assumption that $(0,1)$ is a vertex of $Q^*$ implies that $g=1$, as desired.
\end{proof}

\begin{remark}\label{rem:quotient}
Notice that if
$$H:=\begin{pmatrix}l&i\\0&1\end{pmatrix}\quad\text{ and }\quad K:=\begin{pmatrix}l&j\\0&1\end{pmatrix}$$
are the matrices in Corollary~\ref{cor:quotient} yielding $P$ and $lP^\dual$ then, by equation~\eqref{eq:euclid}, $i j\equiv -1\modb{l}$.
\end{remark}

Since there are only sixteen non-isomorphic reflexive polygons (see, for example,~\cite{PRV00}), this gives a very rapid algorithm for classifying $l$-reflexive polygons. It also explains why Corollary~\ref{cor:phi} holds.

Let us prove Theorem~\ref{thm:number12} in the more general setting of $l$-reflexive loops. Note that $l$-reflexive loops have a well-defined \emph{winding number} $w(P) \in \Z$ (see \cite{PRV00} in the case of $1$-reflexive loops; then apply Corollary~\ref{cor:reflexive}).

\begin{cor}\label{cor:12-thm-loop}
Let $P$ be an $l$-reflexive loop. Then the sum of the length of $P$ and the length of $lP^\dual$ equals $12\, w(P)$.
\end{cor}
\begin{proof}
Let $Q$ be the associated $1$-reflexive loop with respect to the lattice $L := \LambdaP$. By~\cite{PRV00} we know that the desired statement holds for the pair $Q$, $Q^\dual$. Let $b_N(P)$ denote the number of boundary lattice points of an $l$-reflexive loop $P$. By definition $b_L(Q) = b_N(P)$, and by Proposition~\ref{prop:lattice} and Corollary~\ref{cor:reflexive}, $b_{L^\dual}(Q^\dual) = b_{l L^\dual}(l Q^\dual) = b_{\Lambda_{l P^\dual}}(l P^\dual) = b_M(l P^\dual)$.
\end{proof}

It would be interesting to prove Corollary~\ref{cor:12-thm-loop} directly by generalising the proof for $1$-reflexive loops as given in~\cite{PRV00}.

Next, we will prove Proposition~\ref{prop:odd} which states that there are no $l$-reflexive polygons of odd index $l$. There is some experimental evidence that this statement should also hold for $l$-reflexive loops. Unfortunately, we do not know yet how to generalise it to the non-convex setting.

\begin{proof}[Proof of Proposition~\ref{prop:odd}]
Assume $l$ is even. Let $F$ be a facet of $P$. We may assume that its vertices are given as $(a,l)$ and $(b,l)$. Since the vertices of $P$ are primitive, $a$ and $b$ are odd. Therefore, the midpoint $\frac{(a,l)+(b,l)}{2}$ of the facet $F$ is a lattice point. 

By symmetry, this shows that any facet of $P$ and of $l P^\dual$ contains an interior lattice point. By Corollary~\ref{cor:reflexive}, this property also holds for $Q$ and $Q^\dual$, where $Q$ is the associated $1$-reflexive polygon. However, by inspecting the list of sixteen isomorphism classes of $1$-reflexive polygons we see that this is not possible.
\end{proof}

\subsection{Fake weighted projective space}
Let $P\subseteq\NQ$ be an $n$-dim\-ensional lattice simplex containing the origin, with primitive vertices $v_0,v_1,\ldots,v_n$. Suppose further that $(\lambda_0,\lambda_1,\ldots,\lambda_n)$ is a positive collection of weights such that $\gcd{\lambda_0,\lambda_1,\ldots,\lambda_n}=1$ and $\lambda_0v_0+\lambda_1v_1+\ldots+\lambda_nv_n=\orig$. Then the projective toric variety associated with the spanning fan of $P$ is called a \emph{fake weighted projective space with weights $(\lambda_0,\lambda_1,\ldots,\lambda_n)$}~\cite{Buc08,Kas08b}.

A familiar example is $\Proj^2/(\Z/3)$, where $\Z/3$ acts via
$$\varepsilon:x_i\mapsto\varepsilon^ix_i\qquad\text{ for }(x_1,x_2,x_3)\in\Proj^2,$$
where $\varepsilon$ is a third root of unity. The corresponding fan has rays generated by $(2,-1),$ $(-1,2),$ and $(-1,-1)$. The convex hull of these generators is a triangle, and their sum is $\orig$; the variety is a fake weighted projective surface with weights $(1,1,1)$.

Let $P\subseteq\NQ$ be the simplex associated with some fake weighted projective space, and let $\LambdaV{P}$ be the sublattice generated by the vertices $\V{P}$ of $P$. A crucial invariant is the index of $\LambdaV{P}$, which we call the \emph{multiplicity of $P$} and denote $\mult{P}$. The following result summarises some of the properties of $\mult{P}$:

\begin{thm}\label{thm:summary_fwps}
Let $X$ be a fake weighted projective space with weights $(\lambda_0,\lambda_1,\ldots,\lambda_n)$, and let $P$ be the associated simplex in $\NQ$. Let $Q\subseteq\NQ$ be the simplex corresponding to $\Proj(\lambda_0,\lambda_1,\ldots,\lambda_n)$. Then:
\begin{enumerate}
\item~\emph{\cite[Proposition~2]{BB92}:} $X=\Proj(\lambda_0,\lambda_1,\ldots,\lambda_n)$ if and only if $\mult{P}=1$.
\item~\emph{\cite[Proposition~4.7]{Con02} and \cite[Theorem~4.8]{Buc08}:} $X$ is the quotient of the weighted projective space $\Proj(\lambda_0,\lambda_1,\ldots,\lambda_n)$ by the action of the finite group $N/\LambdaV{P}$ acting free in codimension one. In particular $\pi_1^1(X)=N/\LambdaV{P}$.
\item~\emph{\cite[Theorem~4.4]{Con02}:} There exists a Hermite normal form $H$ with determinant $\mult{P}$ such that $P=QH$ (up to the action of $GL(n,\Z)$).
\item~\emph{\cite[Corollaries~2.4 and~2.11]{Kas08b}:} If $X$ has at worst canonical singularities then $\Proj(\lambda_0,\lambda_1,\ldots,\lambda_n)$ has at worst canonical singularities and
$$\mult{P}\le\frac{h^{n-1}}{\lambda_1\lambda_2\ldots\lambda_n},\quad\text{ where }h:=\sum_{i=0}^n\lambda_i.$$
\item~\emph{\cite[Corollary~2.5]{Kas08b} and~\cite[Proposition~5.5]{Con02}:} If $X$ is Gorenstein (equiv.~$P$ is $1$-reflexive) then $\Proj(\lambda_0,\lambda_1,\ldots,\lambda_n)$ is Gorenstein (equiv.~$Q$ is $1$-reflexive) and $\mult{P}\mid\mult{Q^\dual}$.
\end{enumerate}
\end{thm}

Let $Q$ be the polytope associated to weighted projective space $\Proj(\lambda_0,\lambda_1,\ldots,\lambda_n)$. Then the local indices $l_i\mid\lambda_i$, for $0\leq i\leq n$. But the $\lambda_i$ are coprime, hence $\gcd{l_0,l_1,\ldots,l_n}=1$. This gives the following:

\begin{lemma}\label{lem:no_l_reflexive_wps}
Let $P\subset\NQ$ be an $l$-reflexive polygon corresponding to $\Proj(\lambda_0,\lambda_1,\ldots,\lambda_n)$. Then $l=1$ and $\Proj(\lambda_0,\lambda_1,\ldots,\lambda_n)$ is Gorenstein.
\end{lemma}

Of course this is no longer true for fake weighted projective space. For example, the triangle with vertices $\{(-7,-10),$ $(2,5),$ $(1,0)\}$ is $5$-reflexive, and corresponds to the fake weighted projective surface $\Proj(1,2,3)/(\Z/15)$.

There are clear parallels between these important properties of fake weighted projective space and $l$-reflexive polygons. Let $P\subseteq\NQ$ be an $l$-reflexive polygon, let $\LambdaP$ denote the lattice generated by the edges of $P$, and let $Q$ denote the restriction of $P$ to $\LambdaP$. Let $X(P)$ be toric variety generated by the spanning fan of $P$. Corollaries~\ref{cor:reflexive} and~\ref{cor:quotient} can be summarised as follows:
\begin{enumerate}
\item $X(P)$ is Gorenstein if and only if $[N:\LambdaP]=1$.
\item $X(P)$ is the quotient of the Gorenstein surface $X(Q)$ by the action of the finite group $N/\LambdaP$.
\item There exists a Hermite normal form $H$ with determinant $[N:\LambdaP]$ such that $P=QH$ (up to the action of $GL(2,\Z)$).
\end{enumerate}

Proposition~\ref{prop:lattice} tells us that $[N:\LambdaP]=l$, hence we have:

\begin{lemma}
Let $P\subset\NQ$ be an $l$-reflexive polygon corresponding to some fake weighted projective surface. Then $l\mid\mult{P}$.
\end{lemma}

This need not be true in higher dimensions (since the index of the sublattice generated by the facets of $P$ need only divide $l$), although we know of no counterexample.

\subsection{Invariants of $l$-reflexive polygons}\label{sec:invariants}
In this section we shall discuss the Ehrhart $h^*$-vector (also known as the Ehrhart $\delta$-vector) and the order of an $l$-reflexive polygon. We begin by noting that, irrespective of dimension, the $h^*$-vector of $P^\dual$ is palindromic~\cite{FK08}. However the $h^*$-vector of $P$ is palindromic if and only if $l=1$.

By Corollary~\ref{cor:reflexive} any statement about boundary lattice points or vertices of a $1$-reflexive polygon also holds for $l$-reflexive polygons. For instance, it is clear that $l$-reflexive polygons have at most $6$ vertices. In any dimension, the normalised volume of an $l$-reflexive polytope is related to the boundary volume via $\Vol{P}=l\Vol{\bd P}$. In the two dimensional case, the normalised volume can be easily computed as $\Vol{P}=lb$, where $b$ is the number of boundary lattice points of $P$ (and, in particular, $3\leq b\leq 9$).

Let $i$ be the number of interior lattice points of $P$. Pick's Theorem yields
$$l b = \Vol{P} = 2 i + b - 2.$$
Hence, $i = \frac{l-1}{2} \ b + 1$. In particular, 
$$\abs{P \cap N} = b + i = \frac{l+1}{2} \ b + 1.$$
Therefore the generating function $\mathrm{Ehr}_P$ enumerating the number of lattice points in multiples of $P$ (see~\cite{Ehr77,Sta80,Sta97}) can be expressed as a rational function of the form
$$\mathrm{Ehr}_P(t):=\sum\limits_{k \geq 0} \card{(k P) \cap N} \, t^k = \frac{h^*_P(t)}{(1-t)^3},$$
where
$$h^*_P(t) = 1 + (\abs{P \cap N} - 3) t + i t^2 = 1 + \left(\frac{l+1}{2} b - 2\right) t + \left(\frac{l-1}{2} b + 1\right) t^2.$$

Let $L_P(m):=\abs{mP\cap N}$ denote the number of lattice points in $P$ dilated by a factor of $m\in\Z_{\ge 0}$. This is known to be a polynomial of degree $d:=\dim{P}$, called the Ehrhart polynomial. The roots of $L_P$ (regarded as a polynomial over $\C$) have been the subject of much study. In particular:

\begin{thm}[\protect{\cite[Theorem~1.5]{HK10b}}, Golyshev's Conjecture]
Let $P$ be a reflexive polytope of dimension at most five whose facets are all unimodular simplices. If $z\in\C$ is a root of $L_P(m)$, then $\Re{z}=-1/2$.
\end{thm}

\begin{thm}[\protect{\cite[Proposition~1.8]{BHW07}}]
Let $P$ be a lattice polytope such that for all roots $z\in\C$ of $L_P(m)$, $\Re{z}=-1/2$. Then, up to unimodular translation, $P$ is an reflexive polytope.
\end{thm}

We note the following interesting generalisations to $l$-reflexive polygons:

\begin{prop}\label{prop:dim_2_Reimannian}
Let $P$ be an $l$-reflexive polygon not isomorphic to the convex hull of $\{(-1,-1),$ $(-1,2),$ $(2,-1)\}$ (see Figure~\ref{fig:dual_P2}). If $z\in\C$ is a root of $L_P(m)$, then $\Re{z}=-1/(2l)$.
\end{prop}
\begin{proof}
Recall that, in general, if $L_P(m)=c_dm^d+\ldots+c_1m+c_0$ then $c_d=\frac{1}{d!}\Vol{P}$, $c_{d-1}=\frac{1}{2(d-1)!}\Vol{\bd P}$, and $c_0=1$~\cite{Ehr67}. Hence
$$L_P(m)=\frac{lb}{2}m^2+\frac{b}{2}m+1.$$
Let $z\in\C$ be a root of $L_P$. We get
$$z=-\frac{1}{2l}\pm\frac{\sqrt{b^2-8lb}}{2lb}.$$
Since $b^2-8lb\leq 0$ for all $3\leq b\leq 9$ and $l\ge 1$ with the exception of $b=9,$ $l=1$, the result follows.
\end{proof}

\begin{figure}[htbp]
\centering
\includegraphics[scale=0.8]{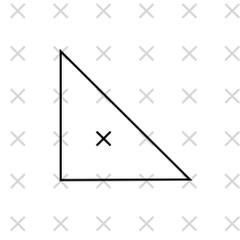}
\caption{The unique exception to Proposition~\ref{prop:dim_2_Reimannian}.}
\label{fig:dual_P2}
\end{figure}

\begin{prop}\label{prop:BHW07_generalised}
Let $P$ be an LDP-polygon of index $l$ such that for all roots $z\in\C$ of $L_P(m)$, $\Re{z}=-1/(2l)$. Then $P$ is an $l$-reflexive polygon.
\end{prop}
\begin{proof}
Let $-1/(2l)\pm\alpha i$ be the two roots of $L_P(m)$. Then
\begin{align*}
L_P(m)&=\beta\left(m+\frac{1}{2l}+\alpha i\right)\left(m+\frac{1}{2l}-\alpha i\right)\\
&=\beta m^2+\frac{\beta}{l}m+\beta\left(\frac{1}{4l^2}+\alpha^2\right),
\end{align*}
hence $\beta=(1/2)\Vol{P}$, $\beta\left(1/(4l^2)+\alpha^2\right)=1$, and
$$\Vol{P}=l\Vol{\bd P}.$$
Let $F\in\F{P}$ be an edge, and let $l_F$ be the corresponding local index. The above result tells us that
$$\sum_{F\in\F{P}}(l-l_F)\Vol{F}=0.$$
But $l_F\le l$ for all $F\in\F{P}$, hence $l_F=l$ and so $P$ is $l$-reflexive.
\end{proof}

In~\cite{KKN08} the \emph{order} $o_P$ of a lattice polytope $P \subseteq \NQ$ containing $\orig$ in its interior was defined in the following way: 
$$o_P := \min \left\{k \in \N \;:\; \intr{P/k} \cap N = \left\{\orig\right\}\right\}.$$
As mentioned in Subsection~\ref{sec:finite}, we have $\intr{P/l} \cap N = \left\{\orig\right\}$ for any $l$-reflexive polytope $P \subseteq \NQ$. Hence $o_P \leq l$. In dimension two we can give a sharp upper bound on this invariant. It seems much harder to find a good lower bound for $o_P$. Table~\ref{tab:orders} lists the orders for the $l$-reflexive polygons of index less than $30$.

\begin{prop}\label{prop:order}
Let $P\subseteq\NQ$ be an $l$-reflexive polygon. Then $o_P\leq(l+1)/2$. 
\end{prop}
\begin{proof}
If $l=1$, then $o_P = 1$. So, let $l > 1$ and assume $x \in \intr{2 P/(l+1)} \cap N$ with $x\ne\orig$. Let $F \in \F{P}$ be such that $x \in \conv{\orig,F}$. Then $\pro{u_F}{x} < 2 l / (l+1) < 2$, hence $\pro{u_F}{x}=1$. Since the vertices of $F$ are primitive we get $l x \in \intr{F} \cap N$, and since $\orig \in \intr{P}$ there exists some $v \in \V{P}$ with $\pro{u_F}{v} < 0$. Therefore $v + l x\ne\orig$ (since $v$ is primitive). By Corollary~\ref{cor:reflexive} we can apply Lemma~4.1(i) in~\cite{Nil05} to the pair $v,lx$ of boundary lattice points. As a consequence, there exists some vertex $z \in \V{F}$ such that $v$ and $z$ lie in a common facet $F' \in \F{P}$. Moreover, either $z = a v + lx$ or $z=v + a lx$ for some $a \geq 1$. Since the first case would imply $\pro{u_F}{v} = 0$, we have $z=v + a lx$. This implies that $v + x \in F' \cap N$. However, Proposition~\ref{prop:lattice} tells us that $l$ divides $\pro{u_F}{v}$, in addition to $\pro{u_F}{v+x} = \pro{u_F}{v} + 1$; a contradiction.
\end{proof}

Note that $o(P_l) = (l+1)/2$, since $(1,0) \in \frac{2}{l} P_l \cap N$.

\begin{table}[htdp]
\caption{The orders $o_P$ of the $l$-reflexive polygons up to index $30$.}
\centering
\begin{tabular}{|r|l||r|l|}\hline
$l$&$o_P$&$l$&$o_P$\\\hline
$1$&$1$&$17$&$4,5,6,9$\\
$3$&$2$&$19$&$3,4,5,7,10$\\
$5$&$2,3$&$21$&$5,11$\\
$7$&$2,3,4$&$23$&$4,5,6,8,12$\\
$9$&$5$&$25$&$4,7,9,13$\\
$11$&$3,4,6$&$27$&$6,14$\\
$13$&$3,4,5,7\phantom{000}$&$29$&$5,6,8,10,15$\\
$15$&$8$&&\\\hline
\end{tabular}
\label{tab:orders}
\end{table}

\subsection{$3k$-reflexive polygons}\label{sec:3k_reflexive}
\begin{prop}\label{prop:3k_reflexive}
Let $P$ be a $3k$-reflexive polygon, where $k$ is an odd positive integer. Then $P$ can be obtained from the $1$-reflexive hexagon $Q:=\mathrm{conv}\{\pm(0,1),$ $\pm(1,1),$ $\pm(1,0)\}$ (in the sense of Corollary~\ref{cor:quotient}). Furthermore, $P\cong 3kP^\dual$.
\end{prop}
\begin{proof}
Let $P:=QH$ be a $3k$-reflexive polygon, where
$$H:=\begin{pmatrix}3k&i\\0&1\end{pmatrix},\quad\gcd{3k,i}=1.$$
On the dual side, by Remark~\ref{rem:quotient} we have that $3kP^\dual\cong Q^\dual K$ where
$$K:=\begin{pmatrix}3k&j\\0&1\end{pmatrix},\quad\gcd{3k,j}=1,\quad ij\equiv-1\modb{3k}.$$
Hence if $i\equiv 1\modb{3}$ then $j\equiv 2\modb{3}$, and if $i\equiv 2\modb{3}$ then $j\equiv 1\modb{3}$. We shall consider the sixteen possible choices for $Q$, and exclude all but the self-dual hexagon.

\begin{enumerate}
\item
Suppose (after possible change of basis) that the vertices of $Q$ include the points $(0,1),$ $(2,1),$ and $(-1,-1)$ (i.e.~$Q$ contains the triangle associated with $\Proj(1,1,2)$). Then $(6k,2i+1)$ and $(-3k,-i-1)$ are vertices of $P$. But one of these points must be divisible by $3$, and hence is not primitive. This allows us to exclude the first six polygons in Figure~\ref{fig:1reflexive_a}, along with their duals. Up to isomorphism, this excludes all eight polygons depicted in Figure~\ref{fig:1reflexive_a}.
\item
Now suppose that $\V{Q}$ contains $(0,1),$ $(1,1),$ and $(-1,-2)$ (i.e.~the triangle associated with $\Proj^2$). Then $(3k,i+1)$ and $(-3k,-i-2)$ are vertices of $P$. Once again we see that these cannot both be primitive, excluding the first two polygons in Figure~\ref{fig:1reflexive_b} and their duals. This excludes the four polygons shown in Figure~\ref{fig:1reflexive_b}.
\item
Let $Q:=\sconv{(0,1),(3,1),(-1,-1)}$ be the triangle associated with $\Proj(1,2,3)$. Then $(-3k,-i-1)$ is a vertex of $P$, forcing $i\equiv 1\modb{3}$. The dual $Q^\dual$ has vertices $\{(-2,1),$ $(0,1),$ $(1,-2)\}$; in particular $(3k,j-2)$ is a vertex of $Q^\dual K$, giving $j\equiv 1\modb{3}$. This is a contradiction.
\item
Consider $Q:=\sconv{\pm(0,1),\pm(1,1)}$ (the polygon associated with $\Proj^1\times\Proj^1$). We see that $QH$ has vertex $(3k,i+1)$, giving $i\equiv 1\modb{3}$. On the dual side, $Q^\dual$ has vertices $\{\pm(0,1),$ $\pm(-2,1)\}$. This gives $(-6k,-2j+1)\in\V{Q^\dual K}$, again forcing $j\equiv 1\modb{3}$. This excludes the final two cases.
\end{enumerate}

The only remaining possibility is that $Q$ is the self-dual hexagon with vertices $\{\pm(0,1),$ $\pm(1,1),$ $\pm(1,0)\}$. We show by direct calculation that $P=QH$ is also self-dual. The vertices of $P$ are given by $\{\pm(0,1),$ $\pm(l,i+1),$ $\pm(l,i)\}$, and the vertices of $lP^\dual=Q^\dual l(H^t)^{-1}$ are $\{\pm(i,-l),$ $\pm(i+1,-l),$ $\pm(1,0)\}$. These are clearly isomorphic.
\end{proof}

\begin{figure}[htbp]
\centering
\subfigure[Case~(i)]{\label{fig:1reflexive_a}\includegraphics[scale=0.8]{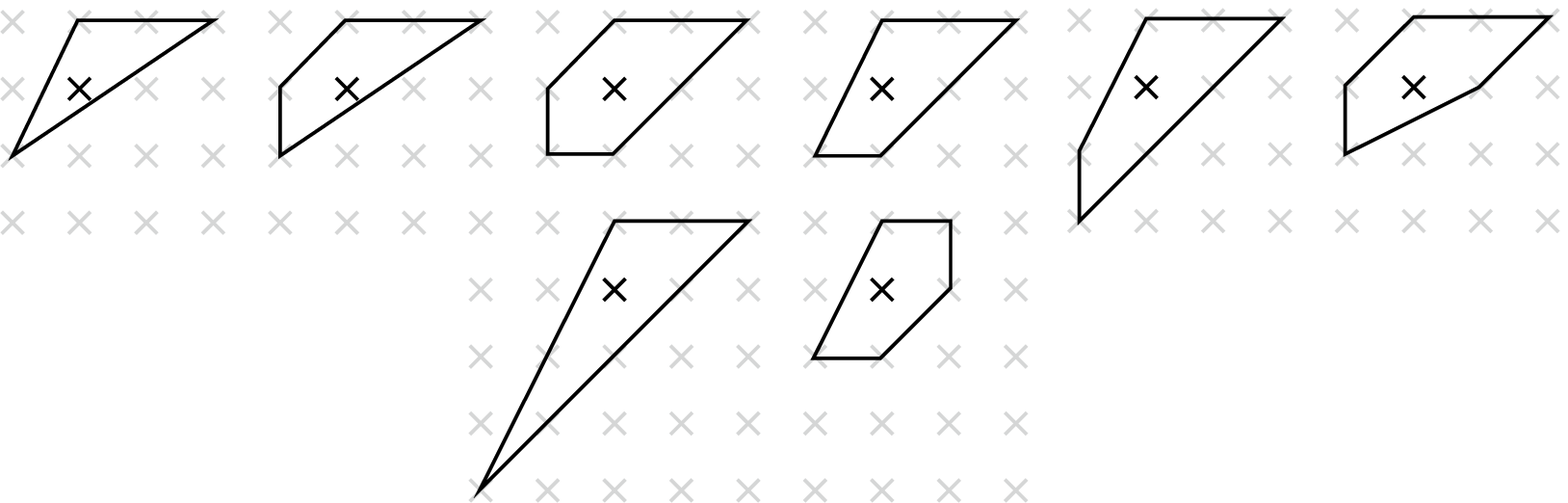}}\\
\subfigure[Case~(ii)]{\label{fig:1reflexive_b}\includegraphics[scale=0.8]{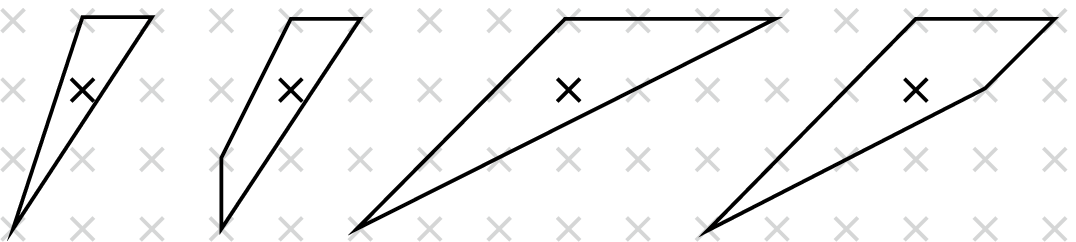}}
\caption{The polygons excluded in cases~(i) and~(ii) of the proof of Proposition~\ref{prop:3k_reflexive}.}
\end{figure}

We conclude this section with one further remark. Clearly any $i$ such that $\gcd{3k,i}=\gcd{3k,i+1}=1$ will give a $3k$-reflexive hexagon, however these need not be distinct. 

Let us make this precise. Note that in order for two such hexagons to be isomorphic it suffices to find a lattice isomorphism mapping the vertices of one edge mutually onto the vertices of another edge. For this we choose two `standard lattices' for each (cone over an) edge of $P$ in such a way that two $3k$-reflexive polytopes $P$ and $P'$ given by $i$ and $i'$ are isomorphic if and only if a standard lattice of an edge of $P$ agrees with a standard lattice of an edge of $P'$. 

Up to central-symmetry we need only consider three edges. We begin with the edge defined by vertices $(3k,i+1)$ and $(0,1)$. By mapping $(3k,i+1)$ to $e_1$ and $(0,1)$ to $e_2$, the lattice $\Z^2$ gets mapped onto 
$$\Z \cdot\frac{1}{3k}(1,-i-1) + \Z \cdot e_2 = \Z \cdot e_1 + \Z\cdot\frac{1}{3k}(h,1),$$
where $h(-i-1)\equiv 1\modb{3k}$. Note that the value of $i$ (respectively, $h$) can be read off uniquely from this `standard lattice'. In the same way, mapping $(3k,i+1)$ to $e_2$ and $(0,1)$ to $e_1$ yields an isomorphism mapping $\Z^2$ onto $\Z \cdot e_1 + \Z \cdot\frac{1}{3k}(-i-1,1) = \Z\cdot\frac{1}{3k}(1,h) + \Z \cdot e_2$, which is just the previous lattice with coordinates switched. Repeating this process for each additional edge, we obtain four more `standard lattices': $\frac{1}{3k}(1,i)+ \Z \cdot e_2 =\Z \cdot e_1+\frac{1}{3k}(-j,1)$ (and its switched lattice), and $\frac{1}{3k}(1,j)+\Z \cdot e_2=\Z \cdot e_1 +\frac{1}{3k}(-i,1)$ (and its switched lattice). Thus we see that two choices $i$ and $i'$ give non-isomorphic hexagons if and only if the sets $\{\pm i,\pm j,-i-1,h\}$ and $\{\pm i',\pm j',-i'-1,h'\}$ are disjoint $\modb{3k}$.

This observation allows very rapid enumeration of the possible $i$. The values for small $k$ are listed in Table~\ref{tab:3k_possible_i}.

\begin{table}[htdp]
\caption{Representatives of the possible choices for $i$ giving non-isomorphism $3k$-reflexive hexagons.}
\begin{center}
\begin{tabular}{|c|l|}
\hline
$k$&$i$\\\hline
$1$&$1$\\
$3$&$1$\\
$5$&$1$\\
$7$&$1,4$\\
$9$&$1,4$\\
$11$&$1,4$\\
$13$&$1,4,16$\\
$15$&$1,7$\\
$17$&$1,4,7$\\
$19$&$1,4,7,10$\\
$21$&$1,4,10$\\
$23$&$1,4,7,19$\\
$25$&$1,7,13$\\\hline
\end{tabular}
\begin{tabular}{|c|l|}
\hline
$k$&$i$\\\hline
$27$&$1,4,7,13,31$\\
$29$&$1,4,7,13,16$\\
$31$&$1,4,7,10,13,25$\\
$33$&$1,4,7,16,28$\\
$35$&$1,16,22$\\
$37$&$1,4,7,10,19,25,31$\\
$39$&$1,4,7,10,16,19$\\
$41$&$1,4,7,10,13,16,25$\\
$43$&$1,4,7,10,19,22,49,52$\\
$45$&$1,7,13,22,31$\\
$47$&$1,4,7,10,16,22,25,40$\\
$49$&$1,4,10,16,19,52,67$\\
$51$&$1,4,7,10,25,28,31,40$\\\hline
\end{tabular}
\end{center}
\label{tab:3k_possible_i}
\end{table}

\section{Examples and open questions}\label{sec:examples}
\subsection{Motivational questions}
Motivated by the positive results in dimension two, there are many natural questions one may ask about $l$-reflexive polytopes in higher dimensions. Which of the results in dimension two extend to higher dimensions? What other properties of $1$-reflexive polytopes can be generalised to higher index? Do the corresponding hypersurfaces have interesting properties -- at least, if the we assume that the ambient space has mild (say, isolated) singularities? What about possible relations to Mirror Symmetry and Calabi-Yau varieties, which spurred the initial interest in $1$-reflexive polytopes~\cite{Bat94}? We remark that Gorenstein polytopes (lattice polytopes where some $r^\text{th}$-multiple is reflexive) may be regarded as being ``$\frac{1}{r}$-reflexive''; they also satisfy a beautiful duality and are related to the construction of mirror-symmetric Calabi-Yau complete intersections~\cite{BB97,BN08}. Can we say something similar about ``$\frac{l}{r}$-reflexive polytopes''?

As we will illustrate below, we cannot expect direct generalisations of most results given in this paper. However, we are convinced that there are many interesting results about (possibly subclasses of) $l$-reflexive polytopes in higher dimensions, and that their study is worthwhile.

\subsection{Any Gorenstein index is possible in dimension three}

Let $P$ be the tetrahedron with vertices $\{(-l,-1,0),$ $(l,0,-1),$ $(0,1,0),$ $(0,0,1)\}$. Then $P^\dual$ is the convex hull of $\{(-2/l,-1,-1),$ $(2/l,-1,-1),$ $(-2/l,3,-1),$ $(2/l,-1,3)\}$. Therefore $P$ is $l$-reflexive if $l$ is odd, and $l/2$-reflexive if $l$ is even. In particular there exist three-dimensional $l$-reflexive polytopes for any index.

\subsection{The edge lattice and the number 24}\label{sec:edge}
The main results in Section~\ref{sec:dim2} fail to hold in dimension three. Perhaps the simplest counterexample is the tetrahedron $P$ with vertices $\{(1,0,0),$ $(3,4,0),$ $(5,0,8),$ $(-9,-4,-8)\}$. This is a $2$-reflexive polytope with $\LambdaP=N$, hence neither Proposition~\ref{prop:lattice} nor Corollary~\ref{cor:reflexive} generalise to higher dimensions.

The toric variety corresponding to $P$ is a fake weighted projective space, $\Proj^3/(\Z/4\times\Z/8)$. By definition we have that $P$ restricted to $\LambdaV{P}$ is the $1$-reflexive simplex associated with $\Proj^3$. The dual polytope $2P^\dual$ is also $2$-reflexive. In this case the corresponding toric variety is $\Proj^3/(\Z/4)$, which has canonical singularities\footnote{ID $547364$ in the classification of toric canonical Fano threefolds~\cite{Kas08a}; see the online \href{http://grdb.lboro.ac.uk/forms/toricf3c}{Graded Ring Database}.}. Note that $o(2P^\dual) = 1$, while the Gorenstein index is $2$. Such behaviour is not possible in dimension two.

In this example, $P$ is $1$-reflexive with respect to the index four sublattice generated by its edges, denoted $\LambdaE{P}$. This restriction gives the polytope $P'$ with vertices $\{(-9,-2,-4),$ $(1,0,0),$ $(3,2,0),$ $(5,0,4)\}$. Similarly, restricting $2P^\dual$ to the index two sublattice $\LambdaE{2P^\dual}$ yields a $1$-reflexive polytope $Q$ with vertices $\{(-1,1,1),$ $(-1,1,2),$ $(-1,3,1),$ $(3,-5,-4)\}$. It is interesting to note that $P'^\dual\cong Q$.

In~\cite[pg.~185]{Snow05} Haase reported the following result, which he attributed to Dais and is a consequence of~\cite{BD96}:

\begin{thm}[\protect{\cite[Theorem~4.3]{Snow05}}]\label{thm:number24}
Let $P\subset\NQ$ be a three-dimensional reflexive polytope. Then:
$$\sum_{E\in\E{P}}\Vol{E}\cdot\Vol{E^\dual}=24,$$
where $E^\dual$ is the edge in $P^\dual$ corresponding to $E$, and $\Vol{E}:=\abs{E\cap N}-1$.
\end{thm}

Since $P'^\dual\cong Q$, we observe that a reformulation of Theorem~\ref{thm:number24} holds for the $2$-reflexive polytope $P$ and its dual $2P^\dual$, where we understand $E^\dual$ to mean the edge in $2P^\dual$ corresponding to $E$.

In the example above we noted that $P'^\dual\cong Q$. As a consequence, the ``$24$-property'' holds for the $2$-reflexive polytope $P$. There exist, however, $2$-reflexive polytopes for which the $24$-property does not hold. Consider the simplex $T$ associated with $\Proj(1,2,3,6)$, namely the convex hull of $\{(-2,-3,-6),$ $(1,0,0),$ $(0,1,0),$ $(0,0,1)\}$. Let
$$H:=\begin{pmatrix}
4&0&1\\
0&4&3\\
0&0&1
\end{pmatrix}.$$
The resulting polytope $S:=TH$ with vertices $\{(-8,-12,-17),$ $(4,0,1),$ $(0,4,3),$ $(0,0,1)\}$ is $2$-reflexive. If we restrict $S$ to the index two sublattice $\Lambda_{\E{S}}$, the resulting simplex is \emph{not} a reflexive polytope (it corresponds to fake weighted projective space $\Proj(1,2,3,6)/(\Z/2\times\Z/4)$). In this case the $24$-property does not hold; the sum is $28$.

Finally, we consider one further family of examples: The $l$-reflexive polytopes contained in the classification of all three-dimensional canonical Fano polytopes~\cite{Kas08a}. These are $l$-reflexive polytopes with exactly one interior point; the corresponding toric Fano three-folds have at worst canonical singularities. This contains the standard $4319$ $1$-reflexive polytopes, along with five $2$-reflexive polytopes\footnote{IDs $520134$, $544310$, $544353$, $547354$, and $547364$ in the Graded Ring Database.}, two $3$-reflexive polytopes\footnote{IDs $544385$ and $547369$.}, and one $5$-reflexive tetrahedron\footnote{ID $547383$.}. In every case the polytope restricted to its edge lattice is $1$-reflexive, and the $24$-property holds.

\begin{conjecture}
If $P$ is a three-dimensional $l$-reflexive polytope such that $P$ restricted to $\LambdaE{P}$ is isomorphic to a $1$-reflexive polytope $Q$, then $lP^\dual$ restricted to $\LambdaE{lP^\dual}$ is isomorphic to $Q^\dual$. In particular $P$ satisfies the $24$-property.
\end{conjecture}

\begin{remark}
If $P$ is a $1$-reflexive polytope in dimension three then $\LambdaE{P}$ is equivalent to the boundary lattice $\LambdaP$~\cite{HN05}, so this is the natural sublattice to consider. In higher dimensions we would expect to restrict to the lattice generated by the codimension two faces. Furthermore, the results of~\cite{HN05} hold only in dimensions three and higher; dimension two is always a special case.
\end{remark}

\subsection{Classification algorithms in higher dimensions?}

Our classification algorithm in dimension two relies on the fact that for any $l$-reflexive polygon $P$ there exists a $1$-reflexive polygon $Q$ such that $P$ is the image of an integer $2\times 2$-matrix of determinant $l$ (Corollary~\ref{cor:reflexive}). This motivates our main question:\\

\textbf{Question:} Is an $l$-reflexive polytope $P$ $1$-reflexive with respect to the vertex lattice $\LambdaV{P}$?\\

We do not know of a counterexample. However, even if this question has a positive answer, it does not immediately yield a general classification algorithm. Notice that the $2$-reflexive polytope $S$ in Subsection~\ref{sec:edge} is the image under multiplication by a matrix of determinant $16>2$. This shows that it would be necessary to have a bound on the index of the vertex lattice of $l$-reflexive polytopes in dimension $n$, perhaps something analogous to Theorem~\ref{thm:summary_fwps}~(iv). This is not clear even in index $1$.

\begin{ack}
The authors wish to thank the RG Lattice Polytopes at the Freie Universit\"at Berlin, the Computational Algebra Group at the University of Sydney, and Imperial College London, for their hospitality and financial support. The first author is supported by EPSRC grant EP/I008128/1, the second author is supported by the US National Science Foundation (DMS 1102424). This work was supported in part by EPSRC Mathematics Platform grant EP/I019111/1.
\end{ack}

\appendix
\section{{\magma} source code}\label{apdx:source_code}
The following basic {\magma} code can be used to regenerate the classification of $l$-reflexive polygons.

\vspace{1em}
\begin{verbatim}
// Returns true iff P is l-reflexive for some index l. Also returns l.
function is_l_reflexive(P)
    if not IsFano(P) then return false,_; end if;
    l:=GorensteinIndex(P);
    if &and[Denominator(v) eq l : v in Vertices(Dual(P))] then
        return true,l;
    else
        return false,_;
    end if;
end function;

// Compute all non-isomorphic l-reflexive polygons generated by
// the Hermite normal forms with determinant l.
procedure generate_polys(l,~polys)
    if l eq 1 then
        polys[1]:=[PolytopeReflexiveFanoDim2(id) : id in [1..16]];
        return;
    end if;
    polys[l]:=[];
    Hs:=[Matrix(2,2,[l, i, 0, 1]) : i in [1..l-1] | GCD(l,i) eq 1];
    for id in [1..16] do
        P:=PolytopeReflexiveFanoDim2(id);
        for H in Hs do
            Q:=P * H;
            bool,k:=is_l_reflexive(Q);
            if bool and not &or[IsIsomorphic(Q,R) : R in polys[k]] then
                Append(~polys[k],Q);
            end if;
        end for;
    end for;
end procedure;

// The main loop (runs from index 1 to 29, takes approx. 1 minute)
polys:=AssociativeArray(Integers());
for l in [1..29 by 2] do generate_polys(l,~polys); end for;
\end{verbatim}

\bibliographystyle{amsalpha}
\newcommand{\etalchar}[1]{$^{#1}$}
\providecommand{\bysame}{\leavevmode\hbox to3em{\hrulefill}\thinspace}
\providecommand{\MR}{\relax\ifhmode\unskip\space\fi MR }
\providecommand{\MRhref}[2]{%
  \href{http://www.ams.org/mathscinet-getitem?mr=#1}{#2}
}
\providecommand{\href}[2]{#2}

\end{document}